\documentclass[12pt, reqno]{amsart}
\usepackage{fullpage}
\usepackage{amsthm, amstext, amsmath, amssymb}
\usepackage{graphicx}
\usepackage{verbatim}
\usepackage{listings, xcolor}
\usepackage{mathtools}
\usepackage[normalem]{ulem}
\usepackage{tikz, ifthen} 
\usepackage{pgfplots}
\usepackage{enumitem}
\usetikzlibrary{positioning}

\usepackage[backref,colorlinks=true,linkcolor=red,citecolor=blue]{hyperref}
\usepackage[alphabetic,backrefs,lite]{amsrefs}

\usepackage{zref, zref-clever}

\zcsetup{cap=true, nameinlink=false}

\usepackage{tikz-cd}

\usepackage{multirow}
\usepackage{float}

\title{Superelliptic degree sets over Henselian fields}

\author{Alexander Galarraga}
\address{Department of Mathematics,
         University of Washington,
         Seattle, WA, 98195}
\email{agalar@uw.edu}

\author{Alexander Wang}
\address{Department of Mathematics,
         University of Washington,
         Seattle, WA, 98195}
\email{awaw@uw.edu}


\theoremstyle{definition}

\AddToHook{env/definition/begin}{%
\zcsetup{countertype={definition=definition}}}
\newtheorem{definition}{Definition}[section]

\AddToHook{env/example/begin}{%
\zcsetup{countertype={definition=example}}}
\newtheorem{example}[definition]{Example}

\AddToHook{env/remark/begin}{%
\zcsetup{countertype={definition=remark}}}
\newtheorem{remark}[definition]{Remark}

\AddToHook{env/warning/begin}{%
\zcsetup{countertype={definition=warning}}}

\theoremstyle{plain}

\AddToHook{env/thm/begin}{
\zcsetup{countertype={definition=theorem}}}
\newtheorem{thm}[definition]{Theorem}

\AddToHook{env/lem/begin}{
\zcsetup{countertype={definition={lemma}}}}
\newtheorem{lem}[definition]{Lemma}

\AddToHook{env/proposition/begin}{%
\zcsetup{countertype={definition=proposition}}}

\AddToHook{env/corollary/begin}{%
\zcsetup{countertype={definition=corollary}}}

\AddToHook{env/property/begin}{%
\zcsetup{countertype={definition=property}}}

\AddToHook{env/question/begin}{%
\zcsetup{countertype={definition=question}}}

\AddToHook{env/conjecture/begin}{%
\zcsetup{countertype={definition=conjecture}}}

\numberwithin{equation}{section}

\newcommand{\val}{v}
\DeclareMathOperator{\im}{im }
\DeclareMathOperator{\ind}{ind }

\pgfplotsset{compat=1.18}

\begin{document}

\begin{abstract}
    Let $K$ be a discretely valued Henselian field.  Creutz and Viray show that the degree set of a curve $C$ over a $p$-adic field can miss infinitely many multiples of the index of $C$, a phenomenon that cannot occur over finitely generated fields.  For curves $C/K$ with a cyclic cover of $\mathbb{P}^1$ of prime degree, under mild assumptions, we completely characterize how and when this behavior can occur, and give a method for computing degree sets of curves of this type.
\end{abstract}

\maketitle

\section{Introduction}

Given a variety $V/K$ over a field $K$, one would like to determine $V(K)$, the set of $K$-rational points. If it happens that $V(K)$ is empty, one would then like to determine for which extensions $L/K$ is $V(L)$ nonempty. Explicitly, define the degree set of $V/K$ by
\begin{align*}
    \mathcal{D}(V/K) \coloneqq \{\deg_K (P) \mid P \in V, \, P \text{ closed} \}
\end{align*}
where $\deg_K P$ denotes the degree of the residue field \textbf{k}$(P)$ over $K$, and the index of $V/K$ by $\ind(V/K) \coloneqq \gcd \mathcal{D}(V)$. We would like to determine the degree set of $V/K$.

When $K$ is a discretely valued Henselian field, determining the index and the degree set is possible. Gabber, Liu, and Lorenzini have shown that that the index over $K$ depends only on data related to the special fiber \cite{index-special-fiber}*{Theorem 8.2}, and Creutz and Viray showed that the degree set can be computed explicitly from the special fiber of a strict normal crossings model \cite{Degrees-on-Varieties}*{Theorem 1.1}. Thus, provided knowledge of the special fiber, one can compute the degree set. For a large class of curves with affine model $f(x,y) = 0$, \cite{Regular-Models}*{Theorem 1.1} gives an effective method, and for hyperelliptic curves with potential semistable reduction, \cite{models-of-hyperelliptic-curves}*{Theorem 1.2} describes the special fiber in terms of cluster diagrams. The special fiber, however, is difficult to compute in general.

Creutz and Viray give examples of curves $C/K$ with index $1$ having degree set that excludes infinitely many integers.  These degree sets do not arise over finite fields or global fields. Indeed, over finitely generated fields, the degree set contains all sufficiently large multiples of the index (e.g. \cite{new-points}*{Theorem 7.5}). We say that $\mathcal{D}(C/K)$ is \textbf{not cofinite} if the complement of $\mathcal{D}(C/K)$ in $\ind(C/K)\mathbb{N}$ is not finite.  We investigate when this behavior occurs for curves $C/K$ which are ``superelliptic" of prime degree (curves with a cyclic cover $C \to \mathbb{P}^1$ of prime degree $q$). Over fields of characteristic not $q$, superelliptic curves are birational to an affine plane curve of the form $y^q = F(x)$. When $q$ is 2, the curve is hyperelliptic.

Following \cite{cluster-diagrams}, we give an answer in terms of clusters. A \textbf{cluster} is a subset of the roots of $F(x)$ contained in an open disk. As the absolute Galois group $\text{Gal}(K^{\text{sep}}/K)$ acts on the roots of $F(x)$, there is a natural action of $\text{Gal}(K^{\text{sep}}/K)$ on clusters (see \zcref{defn: cluster}). Let $\mathcal{O}(\mathfrak{s})$ denote the orbit of a cluster $\mathfrak{s}$ under $\text{Gal}(K^{\text{sep}}/K)$, let $c_\mathfrak{s}$ be the integer defined in \zcref{lem: valuation of polynomial}, and let $\gamma_\mathfrak{s} \in K$ be defined as in \zcref{lem: shifting argument}. Note that $c_\mathfrak{s}$ and $\gamma_\mathfrak{s}$ are computable by hand from the roots of $F(x)$, see \zcref{ex: computed degree set}.

\begin{thm}\label{thm: main thm}
    Suppose that the residue field of $K$ is algebraically closed and that every root of $F(x)$ is tamely ramified. Then, $\mathcal{D}(C/K)$ is not cofinite if and only if $\val(F(0)) \not \equiv 0 \pmod q$, and for every Galois-invariant cluster $\mathfrak{s}$ of roots of $F(x)$,
    \begin{enumerate}[label=$(\roman*)$]
        \item $|\mathfrak{s}| \equiv 0 \pmod q$ and $c_{\mathfrak{s}} \not \equiv 0 \pmod{q}$, and
        \item if $\gamma_{\mathfrak{s}} \in X_{\mathfrak{s}}$, then $\val(F(\gamma_{\mathfrak{s}})) \not \equiv 0 \pmod{q}$.
    \end{enumerate}
    When these conditions are satisfied, we have that
    \begin{align*}
        \mathcal{D}(C/K) \subseteq q\mathbb{N} \cup \bigcup_{\substack{ \mathfrak{s} \, \text{\rm not Galois} \\ \text{\rm invariant}}} |\mathcal{O}(\mathfrak{s})| \mathbb{N}.
    \end{align*}
\end{thm}

The two main tools used in the proof of \zcref{thm: main thm} are \zcref{lem: valuation of polynomial} and \zcref{lem: no points not zero mod q}.  \zcref{lem: valuation of polynomial} provides a formula to compute $\val(F(x_0))$ as a function of $x_0 \in \overline{K}$ and the nearest cluster, and may be of independent interest.  \zcref{lem: no points not zero mod q} shows that there is an obstruction arising from the valuation to points of arbitrary degree.

The results of \cite{Degrees-on-Varieties} show that \zcref{thm: main thm} must relate to the multiplicities of the components of the special fiber of a regular model with strict normal crossings. Following the construction of \cite{Regular-Models}, for many Galois invariant clusters $\mathfrak{s}$ (specifically, for those that fall in case \ref{case: annulus type} of \zcref{lem: types of S}), there exists a $v$-edge $L_\mathfrak{s}$, and that $\delta_{L_\mathfrak{s}}$ equals 1 if and only if the conditions of \zcref{thm: main thm} are not satisfied. When $\delta_{L_\mathfrak{s}}$ equals 1, there are intersecting copies of $\mathbb{P}^1$ with coprime multiplicity, and thus \cite{Degrees-on-Varieties}*{Theorem 1.1} shows that for some $r$, $\mathbb{N}_{>r} \subseteq \mathcal{D}(C/K)$. We compute the degrees of all points reducing to the $\mathbb{P}^1$'s arising from $L_\mathfrak{s}$ simultaneously in \zcref{lem: congruence conditions imply N>r}. When $C/K$ is not $\Delta_v$-regular, the methods of \cite{Regular-Models} do not give a regular model, see \zcref{rem: dokchitser}. As in \zcref{ex: computed degree set}, we can often compute the degree set even when $C/K$ is not $\Delta_v$-regular.

Restricting briefly to the case of hyperelliptic curves, \cite{models-of-hyperelliptic-curves}*{Theorem 1.2} shows that, if the curve has potential semistable reduction (which is equivalent to every root of $F(x)$ being tamely ramified), the special fiber can be computed from a \textit{cluster diagram} (see \cite{original-cluster-diagrams} or the overview \cite{cluster-diagrams-users-guide}) and thus the degree set depends only on the cluster diagram.  Using an additional quantity $c_{\mathfrak{s}}$, both $|\mathfrak{s}|$ and $\val(F(\gamma_{\mathfrak{s}}))$ can be computed from the cluster diagram (see \zcref{rem: compute v(F(gamma))}). We do not know how $c_\mathfrak{s}$ relates to the cluster diagram.

\begin{thm}\label{thm: nice generate examples theorem}
    Let $K$ be a discretely valued Henselian field with algebraically closed residue field of characteristic $p$. Let $q$ be a prime not equal to $p$, and let $n_1, \ldots, n_\ell$ be positive integers. Then, there exists a superelliptic curve $C/K$ such that
    \begin{align*}
        \mathcal{D}(C/K) = q\mathbb{N} \cup n_1\mathbb{N} \cup \ldots \cup n_\ell \mathbb{N}.
    \end{align*}
\end{thm}

As a special case of \zcref{thm: nice generate examples theorem}, we find that for every pair of odd primes $p$ and $q$, there exists a curve $C/K$ such that $\mathcal{D}(C/K) = p \mathbb{N} \cup q\mathbb{N}$. Thus, the density of $\mathcal{D}(C/K)$ in $\ind(C/K)\mathbb{N}$ can be arbitrarily small.


The authors would like to thank Bianca Viray, Paul Fili, and Carlos Rivera. The first author was supported in part by NSF grant DGE-2140004.



\section{Preliminaries} \label{sec: preliminaries}

\subsection{Notation}\label{subsec: notation}

Throughout the paper, let $K$ be a discretely valued Henselian field with value group $\mathbb{Z}$, let $\val$ be the valuation on $\overline{K}$ normalized with respect to $K$ and associated absolute value $|\cdot|$, uniformizer $\pi$, and residue characteristic $p$.  A curve is a smooth, projective, geometrically integral, dimension $1$ scheme over a field. For a prime $q \neq p$, a superelliptic curve $C/K$ of degree $q$ and genus $g$ is the normalization of the projective closure of an affine variety given by $y^q = F(x)$, where $F(x)$ has no $q$-th powers.  Let $F(x) = a_d\prod_{i = 1}(x - \alpha_i)$ be the factorization over $\overline{K}$.

\subsection{General lemmas}

\begin{lem}\label{lem: X^p - a factors but better}
    Let $p$ be a prime, let $k$ be a field, and let $a \in k^{*}$. The polynomial $h(y) = y^p - a$ is either irreducible or $a \in k^{\times p}$.
\end{lem}
\begin{proof} 
    This follows from \cite{pth-power-irreducible}*{Chapter VI, Theorem 9.1}.
\end{proof}

\begin{lem}\label{lem: sol in interval}
    For integers $a,b,c$, and $r$, let $\rho\colon \frac{1}{r} \mathbb{Z} \to \frac{1}{rb} \mathbb{Z}$ be defined by 
    \begin{align*}
        \rho(x) = \frac{a}{b}x + \frac{c}{b}.
    \end{align*}
    The image of $\rho$ intersects $\frac{1}{r}\mathbb{Z}$ if and only if $\gcd(a,b)$ divides $rc$. If this condition holds, then for any open interval $I$ of length greater than $\frac{b+1}{r}$, $\
    \text{im}(\rho)$ restricted to $I \cap \frac{1}{r}\mathbb{Z}$ intersects $\frac{1}{r}\mathbb{Z}$.
\end{lem}
\begin{proof}
    As $\frac{1}{r}\mathbb{Z}$ and $\frac{1}{rb}\mathbb{Z}$ are isomorphic to $\mathbb{Z}$, we construct a map $\varphi\colon \mathbb{Z} \to \mathbb{Z}$ from $\rho$. Explicitly, let $\theta \colon \mathbb{Z} \to \frac{1}{r}\mathbb{Z}$ be division by $r$, and let $\psi\colon \frac{1}{rb}\mathbb{Z} \to \mathbb{Z}$ be multiplication by $rb$. Define $\varphi$ by $\varphi(x) \coloneqq (\psi \circ \rho \circ \theta)(x) = ax + rc$.  Then, $\text{im} (\rho)$ intersects $\frac{1}{r}\mathbb{Z}$ if and only if $\im (\varphi)$ intersects $b\mathbb{Z}$, as $b\mathbb{Z}$ is the image of $\frac{1}{r}\mathbb{Z}$ under $\psi$. The image of $\varphi$ intersects $b\mathbb{Z}$ precisely when there is a solution to $ax \equiv -rc \pmod{b}$, which is solvable if and only if $\gcd(a,b)$ divides $rc$.  For an interval $I$, if $|I| > \frac{b+1}{r}$, then by the pigeonhole principle, the numerators of $I \cap \frac{1}{r}\mathbb{Z}$ contain a complete set of representatives of congruence classes modulo $b$.
\end{proof}

\subsection{Results about Henselian fields}

The following is a strengthening of Krasner's lemma, provided that one works with an element of $K^t$, where $K^t$ is the tamely ramified closure.

\begin{lem}\label{lem: tamely ramified generalized krasners}
    Let $B$ be an open ball in $\overline{K}$ such that $B \cap K^t$ is nonempty.  There exists $\beta \in B$ such that for all $\alpha \in B$, $K(\beta) \subseteq K(\alpha)$ and $\deg_K(\beta)$ equals the size of the orbit of $B$ under $\text{{\normalfont Gal}}(K^t/K)$.
\end{lem}
\begin{proof}
    This is the one dimensional case of the main theorem of \cite{polydisks}.
\end{proof}

\begin{lem}\label{lem: tamely ramified extensions}
    Assume that the residue field of $K$ is algebraically closed. If $r$ is coprime to the residue characteristic of $K$, then $K(\sqrt[r]{\pi})$ is the unique extension of $K$ of degree $r$.
\end{lem}
\begin{proof}
    This follows from \cite{unique-tamely-ramified-extension}*{Chapter II, Proposition 12}.
\end{proof}

\subsection{Structure of degree sets of general schemes over Henselian fields}\label{sec: degree sets}
We establish two lemmas that are crucial to the computation of degree sets. Both are due to Liu and Lorenzini \cite{new-points}.  We state them here in the form we require for our argument.
\begin{lem}\label{lem: closed under multiple}
    Let $X/K$ be a geometrically integral and smooth scheme of finite type and positive dimension. Let $m \in \mathbb{N}$. If $m \in \mathcal{D}(X/K)$, then $m\mathbb{N} \subseteq \mathcal{D}(X/K)$.
\end{lem}
\begin{proof}
    Since $K$ is Henselian, $K$ is a large field \cite{henselian-=>-large}*{Theorem 1.1}. If $m \in \mathcal{D}(X/K)$, there exists some extension $L/K$ of degree $m$ such that $X(L)$ is nonempty, and hence contains a smooth point. From \cite{new-points}*{Proposition 8.3}, we find that any separable extension $M/L$ contains infinitely many new points of $X$. We can then pick a separable extension $M_n/L$ of degree $n$ for each $n \in \mathbb{N}$, as for example, if $\pi$ is a uniformizer for $K$, then $X^n + \pi X + \pi$ will be an irreducible and separable polynomial. Thus, we find that $m\mathbb{N} \subseteq \mathcal{D}(X/K)$. 
\end{proof}

\begin{lem}\label{lem: degree set is dense}
    Let $X/K$ be a geometrically integral and smooth scheme of finite type and positive dimension. Let $W$ be a Zariski open subset of $X$. Then, $\mathcal{D}(W/K) = \mathcal{D}(X/K)$.
\end{lem}
\begin{proof}
    As $W$ embeds into $X$, we have that $\mathcal{D}(W/K) \subseteq \mathcal{D}(X/K)$.  If $m \in \mathcal{D}(X/K)$, then there exists some $L/K$ of degree $m$ such that $X(L)$ is nonempty. Since $K$ is Henselian, $L$ is also Henselian, and hence is a large field \cite{henselian-=>-large}*{Theorem 1.1}, so that $X(L)$ is Zariski dense. Thus, $W(L)$ is nonempty, and hence there exists some point of $W$ having degree dividing $m$. Applying \zcref{lem: closed under multiple}, we find that $m \in \mathcal{D}(W/K)$.
\end{proof}

\section{The valuation of a polynomial}

In this section we address the question of determining, for any $x_0 \in \overline{K}$, the valuation of $F(x_0)$ in terms of the valuation of $x_0$.  We give a complete answer in terms of the roots $\{\alpha_i\}$ of the polynomial $F(x)$.  The following definition is \cite{cluster-diagrams}*{Definition 1.1}.

\begin{definition}\label{defn: cluster}
    A \textbf{cluster} is a nonempty subset $\mathfrak{s} \subseteq \{\alpha_i\}$ of the form $\mathfrak{s} = D \cap \{\alpha_i\}$ for some disk $D_{z,\, d} = \{x \in \overline{K}\mid |x - z|\leq d\}$ for some $z \in \overline{K}$ and $d \in \mathbb{R}$.  Let $\mathcal{O}(\mathfrak{s})$ denote the orbit of $\mathfrak{s}$ as a set under $\text{Gal}(K^{\text{\rm sep}}/K)$, and let $\mathfrak{s}^c$ denote the complement of $\mathfrak{s}$ in $\{\alpha_i\}$.
\end{definition}

\begin{lem}\label{lem: S is a cluster}
    Let $x_0 \in \overline{K}$. Let $R_{x_0} = \min_{i} |x_0 - \alpha_i|$ and let $S$ be the set of roots that achieve the minimum. Then, $S$ is a cluster.
\end{lem}
\begin{proof}
    Choose $z = x_0$ and $d = R_{x_0}$.  We then have that $S = D_{z,d} \cap \{\alpha_i\}$.
\end{proof}

\begin{definition}
    Let $X_\mathfrak{s} \subseteq \overline{K}$ be the set
    \begin{align*}
        X_{\mathfrak{s}} = \{x \in \overline{K} \mid \forall \alpha \in \mathfrak{s}, |x - \alpha| = R_{x}, \forall \alpha' \not\in \mathfrak{s}, |x - \alpha'| > R_x \}.
    \end{align*}
\end{definition}

We think of $X_\mathfrak{s}$ as partitioning $\overline{K}$ by the roots of minimal distance.  By construction, the collection of $\{X_\mathfrak{s}\}$ covers $\overline{K}$ as $\mathfrak{s}$ varies over all clusters.
\begin{lem}\label{lem: valuation of polynomial}
    Fix a cluster $\mathfrak{s}$. There exists $c_{\mathfrak{s}} \in \mathbb{Q}$ such that for all $x_0 \in X_{\mathfrak{s}}$ and any $\alpha \in {\mathfrak{s}}$,
    $$
    \val(F(x_0)) = |\mathfrak{s}|\cdot\val(x_0 - \alpha) + c_{\mathfrak{s}},
    $$
    where $c_{\mathfrak{s}}$ is given by
    \begin{align*}
        c_{\mathfrak{s}} &= \val(a_d) + \sum_{\beta \in {\mathfrak{s}^c}} \val(\alpha - \beta)
    \end{align*}
    and all expressions above are independent of the choice of $\alpha \in {\mathfrak{s}}$.
\end{lem}
Note that for any choice of $\alpha \in {\mathfrak{s}}$, the expression $|\mathfrak{s}|\val(x_0 - \alpha)$ is equivalent to $\sum_{\alpha' \in {\mathfrak{s}}} \val(x_0 - \alpha')$, since all roots in ${\mathfrak{s}}$ are equidistant from $x_0$.  This definition is more symmetric in ${\mathfrak{s}}$ but less useful computationally.

\begin{remark}\label{rem: slope formula}
     By \zcref{lem: valuation of polynomial}, $|\mathfrak{s}|$ and $c_\mathfrak{s}$ are the coefficients of the piecewise affine linear function given by the slope formula, see \cite{baker-slope-formula}*{Theorem 5.15}.
\end{remark}

\begin{proof}[Proof of \zcref{lem: valuation of polynomial}]
    For $\beta \in \mathfrak{s}^c$, 
    $$ R_{x_0} < |x_0-\beta| = |x_0 - \alpha + \alpha - \beta | \leq \max\{ |x_0 - \alpha| , | \alpha - \beta | \} = \max\{R_{x_0},| \alpha - \beta |\}. $$
    We must have that $|\alpha - \beta| > R_{x_0}$ and thus by strong triangle inequality, $|x_0 -\beta| = |\alpha - \beta|$ and we have the same equality on valuations.  Thus, we have
    \begin{align*}
        \val(F(x_0)) &= \val(a_d) + \sum_{\alpha' \in \mathfrak{s}} \val(x_0 - \alpha') +  \sum_{\beta \in \mathfrak{s}^c} \val(x_0 - \beta) = \val(a_d) + \sum_{\alpha' \in \mathfrak{s}} \val(x_0 - \alpha) +  \sum_{\beta \in \mathfrak{s}^c} \val(\alpha - \beta) \\
        &= |\mathfrak{s}| \cdot \val(x_0 - \alpha) + \left(\val(a_d) + \sum_{\beta \in \mathfrak{s}^c} \val(\alpha - \beta)\right)
    \end{align*}
    which gives the desired formula for $x_0 \in X_{\mathfrak{s}}$ (note that the choice of $\alpha$ was arbitrary).
\end{proof}

\section{Properties of clusters}

\begin{lem}\label{lem: types of S}
    Let $\mathfrak{s}$ be a cluster. Then, at least one of the following is true:
    \begin{enumerate}[label=$(\alph*)$]
        \item\label{case: annulus type} there exists some root $\alpha_{\min}$ of $\mathfrak{s}$ such that $\mathfrak{s} = \{ \alpha \mid \val(\alpha) \geq \val(\alpha_{\min})\}$, or
        \item\label{case: ball type} all roots of $\mathfrak{s}$ are contained in $B_{|\alpha|}(\alpha)$ for any $\alpha \in \mathfrak{s}$, and further $X_{\mathfrak{s}} \subseteq B_{|\alpha|}(\alpha)$.
    \end{enumerate}
\end{lem}
\begin{proof}
    Let $\alpha_{\text{min}}$ denote a root of $\mathfrak{s}$ of minimal valuation. Suppose that \ref{case: annulus type} does not hold, that is, there exists some $\beta \in \mathfrak{s}^c$ such that $|\beta| \leq |\alpha_\text{min}|$. Let $x \in X_{\mathfrak{s}}$. 
    As $\beta \in \mathfrak{s}^c$, by definition of $X_{\mathfrak{s}}$, $|x - \beta| > |x - \alpha_\text{min}|$. If $|x| \neq |\alpha_\text{min}|$, then $|x - \alpha_\text{min}| = \max\{|x|, |\alpha_\text{min}| \} \geq \max \{ |x|, |\beta| \} \geq |x - \beta|$, a contradiction, and thus $|x| = |\alpha_\text{min}|$.
    As $|\beta| \leq |\alpha_\text{min}| = |x|$, the strong triangle inequality gives that $|x - \beta|  \leq |x|$, and hence $|x -\alpha_\text{min}| < |x - \beta| \leq |x| = |\alpha_\text{min}|$, so that $x \in B_{|\alpha_{\min}|}(\alpha_{\min})$. For all $\alpha \in \mathfrak{s}$, by the definition of $X_{\mathfrak{s}}$, $|x - \alpha| = |x - \alpha_\text{min}| < |\alpha_\text{min}|$. Hence,
    \begin{align*}
        |\alpha_\text{min} - \alpha | = |\alpha_\text{min} - x + x - \alpha| \leq \max \{|x - \alpha_\text{min}|, |x - \alpha|\} < |\alpha_\text{min}|.
    \end{align*}
    As all elements of $B_{|\alpha_\text{min}|}(\alpha_\text{min})$ have the same absolute value and any element of the ball can be chosen as the center, the ball $B_{|\alpha|}(\alpha)$ does not depend on choice of $\alpha \in \mathfrak{s}$.
\end{proof}

\begin{lem}\label{lem: shifting argument}
    Suppose that the residue field of $K$ is algebraically closed and that the roots of $F(x)$ are tamely ramified. Let $\mathfrak{s}$ be a Galois-invariant cluster. There exists $\gamma_{\mathfrak{s}} \in K$ and $\alpha_{\min} \in {\mathfrak{s}}$ such that ${\mathfrak{s}} = \{ \alpha \mid \val(\alpha - \gamma_{\mathfrak{s}}) \geq \val(\alpha_{\min} - \gamma_{\mathfrak{s}}) \}$.
\end{lem}
\begin{proof}
    If $|{\mathfrak{s}}| = 1$, then we can take $\gamma_{\mathfrak{s}} = \alpha$ for the unique $\alpha \in {\mathfrak{s}}$, which is an element of $K$ as ${\mathfrak{s}}$ is fixed by the Galois action. Now suppose $|{\mathfrak{s}}| > 1$, and fix some $\alpha \in {\mathfrak{s}}$. Define
    \begin{align*}
        \epsilon_{\mathfrak{s}} := \max_{\alpha' \in {\mathfrak{s}}} \{ |\alpha - \alpha'| \} \qquad \delta_{\mathfrak{s}} := \min_{\beta \in \mathfrak{s}^c} \{|\alpha - \beta|\}.
    \end{align*}
    For the remainder of the proof, let $\alpha' \in {\mathfrak{s}}$ be such that $|\alpha - \alpha'| = \epsilon_{\mathfrak{s}}$, and let $\beta \in \mathfrak{s}^c$ be such that $|\alpha - \beta| = \delta_{\mathfrak{s}}$. Let $x \in X_{\mathfrak{s}}$, so that by the definition of $X_{\mathfrak{s}}$,
    \begin{align*}
        \delta_{\mathfrak{s}} &= |\alpha - \beta| = |\alpha - x + x - \beta| \leq \max \{|\alpha - x|, |x - \beta| \} = |x - \beta|, \\
        \epsilon_{\mathfrak{s}} &= |\alpha - \alpha'| = |\alpha -x + x - \alpha'| \leq \max \{R_x, R_x \} < |x - \beta| = \delta_{\mathfrak{s}}.
    \end{align*}
    Now, let $L/K$ be the splitting field of $F(x)$. As all roots of $F(x)$ are tamely ramified, by \zcref{lem: tamely ramified extensions}, there exists $r \in \mathbb{N}$ and a uniformizer $\tau$ for $L$ such that $\tau^r = \pi$ and $L = K(\tau)$. Expanding $\alpha$, $\alpha'$ and $\beta$ in $\tau$, we have that
    \begin{align*}
        \alpha = \sum_i b_i \tau^i \qquad \alpha' = \sum_i b'_i \tau^i \qquad \beta = \sum_i c_i \tau^i.
    \end{align*}
    Let $N$ be the smallest integer such that $b_N \neq c_N$, so that $\delta_{\mathfrak{s}} = |\alpha - \beta| = |\tau^N|$, and let $N'$ be the smallest integer such that $b_{N'} \neq b'_{N'}$, so that $\epsilon_{\mathfrak{s}} = |\alpha - \alpha'| = |\tau^{N'}|$. We have that $|\tau^{N'}| < |\tau^N|$, so that $N < N'$. Let $\gamma_{\mathfrak{s}} \coloneqq \sum_{-\infty < i}^{N' - 1} b_i\tau^i$.
    Note that by construction of $\gamma_{\mathfrak{s}}$ and definition of $\beta$, for all $\beta' \in \mathfrak{s}^c$ and all $\alpha'' \in {\mathfrak{s}}$,
    \begin{align}\label{eq: distance from gamma}
        \val(\beta' - \gamma_{\mathfrak{s}}) \leq \val(\beta - \gamma_{\mathfrak{s}}) = N \val(\tau) < N'\val(\tau) = \val(\alpha' - \gamma_{\mathfrak{s}}) \leq \val(\alpha'' - \gamma_{\mathfrak{s}}).
    \end{align}
    Thus, $\gamma_{\mathfrak{s}}$ has the desired property.  Finally, we show that $\gamma_{\mathfrak{s}} \in K$. As $L/K$ is tamely ramified and the residue field of $K$ is algebraically closed, $\text{char } K \nmid [L:K]$, and hence $L/K$ is separable and thus Galois. Suppose for the sake of contradiction that $\gamma_{\mathfrak{s}} \not \in K$, so that there exists $\sigma \in \text{Gal}(L/K)$ such that $\sigma(\gamma_{\mathfrak{s}}) \not = \gamma_{\mathfrak{s}}$. As the residue field of $K$ is algebraically closed, for all $i$, $\sigma(b_i) = b_i$. Thus, we see there must exist some minimal $M < N'$ such that $\sigma(\tau^M) \neq \tau^M$ and $b_M \neq 0$. However, this shows that $\sigma(\alpha) \neq \alpha'$, and hence as ${\mathfrak{s}}$ is Galois invariant, there exists $\alpha_{\min} \in {\mathfrak{s}}$ such that $|\alpha - \alpha_{\min}| = |\tau^M| > |\tau^{N'}| = |\alpha - \alpha'| = \epsilon_{\mathfrak{s}}$, contradicting the definition of $\epsilon_{\mathfrak{s}}$.
\end{proof}

\begin{remark}
    For a cluster $\mathfrak{s}$, there must exist some cluster $\mathfrak{s}'$ such that $\gamma_{\mathfrak{s}} \in X_{\mathfrak{s}'}$. From \zcref{eq: distance from gamma}, for any $\alpha \in \mathfrak{s}$ and $\beta' \in \mathfrak{s}^c$, $|\beta' -\gamma_\mathfrak{s}| > |\alpha - \gamma_{\mathfrak{s}}|$, and thus the collection of roots closest to $\gamma_{\mathfrak{s}}$ is a subset of $\mathfrak{s}$, that is, $\mathfrak{s}' \subseteq \mathfrak{s}$. 
    When $\gamma_{\mathfrak{s}}$ is closest to some root $\alpha$, then by Galois-invariance of the valuation, $\val(\gamma_{\mathfrak{s}} - \alpha')$ is constant for every conjugate $\alpha'$ of $\alpha$.  We therefore conclude that $\mathfrak{s}' \subseteq \mathfrak{s}$ must be Galois-invariant, and furthermore if $\mathfrak{s}$ contains no Galois-invariant subclusters, then $\gamma_{\mathfrak{s}} \in X_{\mathfrak{s}}$.
\end{remark}

\begin{lem}\label{lem: cS integer}
    Suppose that the residue field of $K$ is algebraically closed and that the roots of $F(x)$ are tamely ramified. Let $\mathfrak{s}$ be a Galois-invariant cluster. Then $c_{\mathfrak{s}} \in \mathbb{Z}$.
\end{lem}
\begin{proof}
    Note that we can apply a rational change of coordinates: if $\gamma \in K$, consider $F(x + \gamma)$. The set ${\mathfrak{s}}_\gamma=\{ \alpha - \gamma \mid \alpha \in {\mathfrak{s}}\}$ is a Galois-invariant cluster for $F(x+ \gamma)$ and $c_{{\mathfrak{s}}_\gamma} = c_{\mathfrak{s}}$. By \zcref{lem: shifting argument}, we can assume that there exists $\alpha_\text{min}$ such that ${\mathfrak{s}} = \{ \alpha \mid \val(\alpha) \geq \val(\alpha_\text{min})\}$.
    
    As $\val(a_d) \in \mathbb{Z}$, we must show that $\sum_{\beta \in \mathfrak{s}^c} \val(\alpha_\text{min} - \beta) \in \mathbb{Z}$. Both ${\mathfrak{s}}$ and the set of roots $\{\alpha_i\}$ are Galois invariant, so ${\mathfrak{s}}^c$ is as well. For each Galois orbit contained in ${\mathfrak{s}}^c$, pick some representative $\eta$ with minimal polynomial $f_\eta$, and let $\mathcal{M} \subseteq {\mathfrak{s}}^c$ be the set of all $\eta$. The sum $\sum_{\beta \in \mathfrak{s}^c} \val(\alpha_\text{min}-\beta)$ equals $\sum_{\eta \in \mathcal{M}} \val(f_{\eta}(\alpha_\text{min}))$. We show that for all $\eta \in \mathcal{M}$, $\val(f_{\eta}(\alpha_\text{min})) \in \mathbb{Z}$.

    For all $\beta \in \mathfrak{s}^c$, $\val(\beta) < \val(\alpha_\text{min})$, specifically, we have that for all $\eta \in \mathcal{M}$, $\val(\eta) < \val(\alpha_\text{min})$. As $f_\eta$ is irreducible, all roots have the same valuation. Thus, by the strong triangle inequality,
    \begin{align*}
        \val(f_\eta(\alpha_\text{min}))  &= \val \left( \prod_{\beta \mid f_\eta(\beta) = 0} ( \alpha_\text{min} - \beta) \right) \\
        &= \sum_{\beta \mid f_\eta(\beta) = 0} \val(\alpha_\text{min} - \beta) = \sum_{\beta \mid f_\eta(\beta) = 0} \val(\beta) = (\deg f_\eta) \cdot \val(\eta).
    \end{align*}
    As $f_\eta$ is the minimal polynomial of $\eta$, a Newton polygon argument shows that the denominator of $\val(\eta)$ divides $\deg f_\eta$, and hence $\val(f_\eta(\alpha_\text{min})) \in \mathbb{Z}$ as desired.
\end{proof}

\begin{remark}
    Neither \zcref{lem: shifting argument} nor \zcref{lem: cS integer} necessarily hold if the roots of $F(x)$ are wildly ramified. Consider $F(x) = (x^3 - 3)(x^2 - 3)$ over $\mathbb{Q}_3$. Let $\zeta_3 = \frac{-1 + \sqrt{-3}}{2}$ be a root of the second cyclotomic polynomial, $x^2 + x + 1$, so that the roots of $F(x)$ are $\{ \sqrt{3}, - \sqrt{3} , \sqrt[3]{3}, \zeta_3 \sqrt[3]{3}, \zeta_3^2 \sqrt[3]{3} \}$, and let $\mathfrak{s} = \{\sqrt[3]{3},\zeta_3 \sqrt[3]{3}, \zeta_3^2 \sqrt[3]{3}\}$. As $\mathfrak{s} = B_{|\sqrt[3]{3}|}(\sqrt[3]{3}) \cap \{ \alpha_i \}$, $\mathfrak{s}$ is a Galois-invariant cluster.  For all $\gamma \in K$, either $\val(\gamma) \geq 1$, in which case $\val(\alpha - \gamma) = \val(\alpha)$ for all $\alpha$, or $\val(\gamma) \leq 0$, in which case $\val(\alpha-\gamma) = \val(\gamma)$ for all $\alpha$.  Therefore, no $\gamma$ can satisfy the conclusions of \zcref{lem: shifting argument} for $\mathfrak{s}$.  Further,
    \begin{align*}
        c_{\mathfrak{s}} = \val(1) + \val\left(\sqrt[3]{3} - \sqrt{3}\right)  + \val\left(\sqrt[3]{3} + \sqrt{3} \right) = \frac{2}{3}.
    \end{align*}
\end{remark}

\begin{remark}
    The proof of \zcref{lem: cS integer} shows that if there exists $\alpha_{\text{min}} \in \mathfrak{s}$ such that $\mathfrak{s} = \{\alpha \mid \val(\alpha) \geq \val(\alpha_{\text{min}}) \}$, then $c_\mathfrak{s}$ is an integer.
\end{remark}

\section{Degree sets of superelliptic curves}\label{sec: sup curves}

We study the degree set of a superelliptic curve $C/K$ via the affine open $U \subseteq C$ such that $U \simeq V(y^q - F(x))\smallsetminus\{(\alpha_i,0)\}$.  We consider the $x$-coordinate map $x\colon C \to \mathbb{P}^1$.

\subsection{Obstructing points of arbitrary degree} We first use the valuation to show that under certain modular arithmetic conditions, we cannot obtain points of arbitrary degree.

\begin{lem}\label{lem: no points not zero mod q}
    Suppose that the residue field of $K$ is algebraically closed and that every root of $F(x)$ is tamely ramified. Fix a Galois-invariant cluster $\mathfrak{s}$. Let $\gamma_{\mathfrak{s}}$ be constructed from $\mathfrak{s}$ as in \zcref{lem: shifting argument}. Let $r$ be a natural number not divisible by $q$. If there exists $P \in U(\overline{K})$ of degree $r$ with $x(P) \in X_{\mathfrak{s}}$, then one of the following holds:
    \begin{enumerate}[label=$(\roman*)$]
        \item\label{condition: bS not 0 mod q first lemma} $|\mathfrak{s}| \not \equiv 0 \pmod{q}$, or
        \item\label{condition: bS and cS both 0 mod q first lemma} $|\mathfrak{s}| \equiv 0 \pmod q$ and $c_{\mathfrak{s}} \equiv 0 \pmod{q}$, or
        \item\label{condition: constant term 0 mod q first lemma} $0 \in X_{\mathfrak{s}}$ and $\val(F(0)) \equiv 0 \pmod{q}$, or
        \item\label{condition: F(gamma) mod q first lemma} $\gamma_{\mathfrak{s}} \in X_{\mathfrak{s}}$ and $\val(F(\gamma_{\mathfrak{s}})) \equiv 0 \pmod{q}$.
    \end{enumerate}
\end{lem}

\begin{remark}
    The quantities appearing in conditions \ref{condition: bS not 0 mod q first lemma} and \ref{condition: bS and cS both 0 mod q first lemma} in \zcref{lem: no points not zero mod q} are the coefficients appearing in the slope formula, see \zcref{rem: slope formula}.
\end{remark}

\begin{proof}[Proof of \zcref{lem: no points not zero mod q}]
    Let $(x_0,y_0) \in U(\overline{K})$ with $[K(x_0,y_0):K] = r$. By \zcref{lem: X^p - a factors but better}, either $K(x_0, y_0)$ is a degree $q$ extension of $K(x_0)$, or $F(x_0) \in K(x_0)^{\times q}$. The first case is not possible by assumption, so we must have that $F(x_0) \in K(x_0)^{\times q}$ and that $K(x_0, y_0) = K(x_0)$. As the residue field of $K$ is algebraically closed, $K(x_0)$ is totally ramified with value group $\frac{1}{r}\mathbb{Z}$, so that $\val(F(x_0))$ is a multiple of $q$ in $\frac{1}{r}\mathbb{Z}$. As $x_0 \in X_{\mathfrak{s}}$, by \zcref{lem: valuation of polynomial}, for any $\alpha \in \mathfrak{s}$,
    \begin{align*}
        \frac{\val(F(x_0))}{q} = \frac{|\mathfrak{s}|}{q}\val(x_0 - \alpha) + \frac{c_{\mathfrak{s}}}{q}.
    \end{align*}
    We consider the possibilities for $\val(x_0 - \alpha)$, which by strong triangle inequality, are as follows: 
    \begin{enumerate}
        \item $\val(x_0 - \alpha) = \val(x_0) \leq v(\alpha)$ \label{case: pigeonhole}
        \item $\val(x_0 - \alpha) = \val(\alpha) < \val(x_0)$ \label{case: valuation of root}
        \item $\val(x_0 - \alpha) > \val(\alpha) = \val(x_0)$. \label{case: close to root}
    \end{enumerate}

    Case \ref{case: pigeonhole}: First suppose that $\val(x_0 - \alpha) = \val(x_0)$. Then, as $\frac{\val(F(x_0))}{q} \in \frac{1}{r} \mathbb{Z}$, we have that the map $\frac{1}{r}\mathbb{Z} \to \frac{1}{r}\mathbb{Z}$ given by $x \mapsto \frac{|\mathfrak{s}|}{q}x + \frac{c_{\mathfrak{s}}}{q}$ satisfies the conditions of \zcref{lem: sol in interval}, so that $|\mathfrak{s}| \not \equiv 0 \pmod q$ or $|\mathfrak{s}| \equiv 0 \pmod q$ and $c_{\mathfrak{s}} \equiv 0 \pmod q$.
    
    Case \ref{case: valuation of root}: Now assume that $\val(x_0 - \alpha) = \val(\alpha)$ and $\val(x_0) > \val(\alpha)$. As case \ref{case: ball type} of \zcref{lem: types of S} does not hold, case \ref{case: annulus type} must hold, so there exists $\alpha_\text{min}$ such that $\mathfrak{s} = \{\alpha' \mid \val(\alpha') \geq \val(\alpha_\text{min})\}$. As $\val(x_0 - \alpha)$ is independent of the choice of $\alpha \in \mathfrak{s}$, all roots of $\mathfrak{s}$ must have the same valuation, which further must be the maximal valuation among all roots of $F(x)$. Thus, for any root $\alpha'$ of $F(x)$, $\val(x_0) > \val(\alpha')$. We see that $0 \in X_{\mathfrak{s}}$ as $0$ is closest to the roots of maximal valuation, so
    \begin{align*}
        \val(F(x_0)) &= \val(a_d) + \sum_i \val(x_0 - \alpha_i) = \val(a_d) + \sum_i \val(\alpha_i) = \val(F(0)).
    \end{align*}
    Thus, $\val(F(x_0)) \in \mathbb{Z}$. As $\val(F(x_0))$ is a multiple of $q$ in $\frac{1}{r}\mathbb{Z}$ and $q$ does not divide $r$, $\val(F(x_0)) = \val(F(0)) = \val(F(\gamma_{\mathfrak{s}}))$ is a multiple of $q$ in $\mathbb{Z}$.
    
    Case \ref{case: close to root}: Finally, suppose that $\val(x_0 - \alpha) > \val(\alpha) = \val(x_0)$.
    By \zcref{lem: shifting argument}, there exists $\gamma_{\mathfrak{s}} \in K$ and $\alpha_{\text{min}} \in \mathfrak{s}$ such that $\mathfrak{s} = \{\alpha \mid \val(\alpha - \gamma_{\mathfrak{s}}) \geq \val(\alpha_\text{min} - \gamma_{\mathfrak{s}}) \}$. Let $G(x) = F(x + \gamma_{\mathfrak{s}})$, consider the isomorphic curve $y^q = G(x)$, and let $\mathfrak{s}'$ be the $G(x)$-cluster $\{ \alpha - \gamma_S \mid \alpha \in \mathfrak{s}\}$. 
    
    We show that for all $\alpha \in \mathfrak{s}$, $x_0 - \gamma_{\mathfrak{s}}$ is not contained in $B_{|\alpha - \gamma_{\mathfrak{s}}|}(\alpha - \gamma_{\mathfrak{s}})$. If $|\mathfrak{s}| = 1$, then $\gamma_{\mathfrak{s}} = \alpha$ and as $P \in U(\overline{K})$, $x(P)$ does not equal $\alpha$, and hence $x_0 - \gamma_{\mathfrak{s}} \not \in B_0(0)$. Now suppose that $|\mathfrak{s}| > 1$, and let $\alpha' \in \mathfrak{s}$ be such that $|\alpha - \alpha'|$ is maximal as in the proof of \zcref{lem: shifting argument}. By construction of $\gamma_{\mathfrak{s}}$, $|\alpha - \gamma_{\mathfrak{s}}| = |\alpha - \alpha'|$. As $x_0 \in X_{\mathfrak{s}}$, we have that
    \begin{align*}
        |x_0 - \gamma_{\mathfrak{s}} - (\alpha - \gamma_{\mathfrak{s}})| = |x_0 - \alpha| &= |x_0 - \alpha' + \alpha' - \alpha| \\
        &\leq \max \{|x_0 - \alpha'|, |\alpha' - \alpha|\} = \max \{|x_0 - \alpha|, |\alpha - \gamma_{\mathfrak{s}}|\}.
    \end{align*}
    If $|\alpha - \gamma_{\mathfrak{s}}| > |x_0 - \alpha|$, then the strong inequality gives a contradiction. Thus, $|\alpha - \gamma_{\mathfrak{s}}| \leq |x_0 - \alpha|$, and hence $x_0 - \gamma_{\mathfrak{s}} \not \in B_{|\alpha - \gamma_{\mathfrak{s}}|}(\alpha - \gamma_{\mathfrak{s}})$ as desired. 
    Then, $\val((x_0 - \gamma_\mathfrak{s}) - (\alpha - \gamma_\mathfrak{s})) \leq \val(\alpha - \gamma_\mathfrak{s})$, and hence $x_0 - \gamma_{\mathfrak{s}}$ must fall in either case \ref{case: pigeonhole} or case \ref{case: valuation of root} when considering the $G(x)$-cluster $\mathfrak{s}'$.  As $|\mathfrak{s}'| = |\mathfrak{s}|$ and $c_{\mathfrak{s}'} = c_{\mathfrak{s}}$, in case \ref{case: pigeonhole} the congruence conditions hold. Finally, $G(0) = F(\gamma_{\mathfrak{s}})$, and thus in case \ref{case: valuation of root} we have that $\val(F(\gamma_{\mathfrak{s}})) \in q\mathbb{Z}$.
\end{proof}

\begin{remark}
        If $\gamma_{\mathfrak{s}} \neq 0$, for all $\alpha \in \mathfrak{s}$, $v(\alpha) = v(\gamma_{\mathfrak{s}})$.  If there exist $\alpha, \alpha' \in \mathfrak{s}$ such that $\val(\alpha_j) \neq \val(\alpha_{j'})$, then $\gamma_{\mathfrak{s}} = 0$. Thus, for many Galois-invariant ${\mathfrak{s}}$, $\gamma_{\mathfrak{s}}$ is not contained in $X_{\mathfrak{s}}$. 
\end{remark}

\begin{remark}\label{rem: compute v(F(gamma))}
    Define the \textbf{depth} of $\mathfrak{s}$ by $d_\mathfrak{s} \coloneqq \min_{\alpha, \alpha'} \val(\alpha - \alpha')$.  When $|{\mathfrak{s}}| > 1$ and $\gamma_{\mathfrak{s}} \in X_{\mathfrak{s}}$, we have that $\val(\gamma_{\mathfrak{s}} - \alpha_j) = d_{\mathfrak{s}}$, and hence $\val(F(\gamma_{\mathfrak{s}})) = |\mathfrak{s}| d_{\mathfrak{s}} +c_{\mathfrak{s}}$.
\end{remark}

\subsection{Constructing points of large degree} We prove a strengthening of the converse of \zcref{lem: no points not zero mod q}. Using Hensel's lemma, we show that if there exists a point $P \in U(\overline{K})$ with degree not a multiple of $q$, then the degree set contains points of all sufficiently large degrees.

\begin{lem}\label{lem: congruence conditions imply N>r}
    Suppose that the residue field of $K$ is algebraically closed and that every root of $F(x)$ is tamely ramified. Fix a Galois-invariant cluster $\mathfrak{s}$. If one of the following holds:
    \begin{enumerate}[label=$(\roman*)$]
        \item\label{condition: bS not 0 mod q} $|\mathfrak{s}| \not \equiv 0 \pmod{q}$, or
        \item\label{condition: bS and cS both 0 mod q} $|\mathfrak{s}| \equiv 0 \pmod q$ and $c_{\mathfrak{s}} \equiv 0 \pmod{q}$, or
        \item\label{condition: constant term 0 mod q} $0 \in X_{\mathfrak{s}}$ and $\val(F(0)) \equiv 0 \pmod{q}$, or
        \item\label{condition: gamma term 0 mod q} $\gamma_{\mathfrak{s}} \in X_{\mathfrak{s}}$ and $\val(F(\gamma_{\mathfrak{s}})) \equiv 0 \pmod{q}$,
    \end{enumerate}
    then there exists $r_0 \in \mathbb{N}$ such that for all $r \geq r_0$ with $\gcd(r,q) = 1$, there exists a point $P \in U(\overline{K})$ of degree $r$ with $x(P) \in X_{\mathfrak{s}}$.
\end{lem}
\begin{proof}
    Following the argument in the proof of case \ref{case: close to root} of \zcref{lem: no points not zero mod q}, without loss of generality we can take $\gamma_{\mathfrak{s}} = 0$, so condition \ref{condition: gamma term 0 mod q} reduces to condition \ref{condition: constant term 0 mod q}.
    Suppose that condition \ref{condition: constant term 0 mod q} holds, so that $0 \in X_{\mathfrak{s}}$ and hence ${\mathfrak{s}}$ is the set of roots of maximal valuation. We let $r_0 = 1$ and construct points of all degrees $r \geq r_0$. Let $N$ be an integer greater than $\max_{\alpha \in \mathfrak{s}}\{\val(\alpha)\}$, and let $x_0 = \pi^{N + \frac{1}{r}}$. By construction, $x_0 \in X_{\mathfrak{s}}$, and hence by \zcref{lem: valuation of polynomial}, for any $\alpha \in \mathfrak{s}$,
    \begin{align*}
        \val(F(x_0)) = |\mathfrak{s}| \val(x_0 - \alpha) + c_{\mathfrak{s}} = |\mathfrak{s}|\val(\alpha) + c_{\mathfrak{s}} = \val(F(0)) \equiv 0 \pmod q.
    \end{align*}
    Thus, $(x_0,\sqrt[q]{F(x_0)})$ is a $\overline{K}$-point of degree $r$ over $K$.
    
    Now suppose that either condition \ref{condition: bS not 0 mod q} or condition \ref{condition: bS and cS both 0 mod q} holds. We again use the argument in the proof of case \ref{case: close to root} of \zcref{lem: no points not zero mod q} to reduce to the case of $\gamma_{\mathfrak{s}} = 0$.  By \zcref{lem: shifting argument}, there exists $\alpha_{\min}$ such that ${\mathfrak{s}}$ indexes the set $\{ \alpha \mid \val(\alpha) \geq \val(\alpha_{\min}) \}$. Let $N = \val(\alpha_{\min})$ and let $M = \max_{\beta \in \mathfrak{s}^c} \{\val(\beta)\}$. Let $r_0$ be the smallest integer greater than $\frac{q+1}{N-M} + 1$. We construct a point of degree $r$ for all $r \geq r_0$. Consider the affine map $\rho\colon (M,N) \cap \frac{1}{r}\mathbb{Z} \to \frac{1}{rq}\mathbb{Z}$ given by $\rho(x) = \frac{|\mathfrak{s}|}{q}   x + \frac{c_{\mathfrak{s}}}{q}$. By \zcref{lem: sol in interval}, there exists $\frac{a}{r} \in (M,N)$ such that $\rho\left( \frac{a}{r} \right) \in \frac{1}{r}\mathbb{Z}$. Write $\frac{a}{r} = \frac{su}{st}$ with $\gcd(u,t) = 1$, and let $x_0'$ be a root of $x^t - \pi^u$, so that $x_0'$ has degree $t$. Let $\tau$ be a uniformizer for $K(x_0')$ and let $R$ be an integer larger than $s\cdot \frac{u}{t}$ relatively prime to $s$. Let $\varepsilon$ be a root of $x^s - \tau^R$. A Newton polygon argument shows that the degree of $\varepsilon$ over $K(x_0')$ is $s$ and that $\val(\varepsilon) = \frac{R}{s} > \val(x_0')$. Let $x_0 = x_0' + \varepsilon$, so that $K(x_0)$ has degree $st = r$ over $K$, and $\val(x_0) = \val(x_0') = \frac{a}{r}$. As $M < \val(x_0) < N$, $x_0 \in X_{\mathfrak{s}}$, and further
    \begin{align*}
        \frac{\val(F(x_0))}{q} = \frac{|\mathfrak{s}|}{q}\val(x_0 - \alpha) + \frac{c_{\mathfrak{s}}}{q} = \frac{|\mathfrak{s}|}{q} \val(x_0) + \frac{c_{\mathfrak{s}}}{q} = \rho\left(\frac{a}{r}\right) \in \frac{1}{r}\mathbb{Z}.
    \end{align*}
    Thus, $\val(F(x_0))$ is a multiple of $q$ in $\frac{1}{r}\mathbb{Z}$, and hence $(x_0,\sqrt[q]{F(x_0)})$ has degree $r$ over $K$.
\end{proof}

\begin{lem}\label{lem: contains qN}
    The degree set $\mathcal{D}(C/K)$ contains $q\mathbb{N}$.
\end{lem}
\begin{proof}
    We show that either $1 \in \mathcal{D}(C/K)$ (and hence equals $\mathbb{N}$) or that $q \in \mathcal{D}(C/K)$, which is sufficient by \zcref{lem: closed under multiple}. Set $x=0$ and consider $y^q = F(0)$.
    By \zcref{lem: X^p - a factors but better}, either $y^q - F(0)$ is irreducible, or there exists $b \in K$ such that $b^q = F(0)$. If $y^q - F(0)$ is irreducible, $(0, F(0)^{1/q})$ has degree $q$, otherwise, the point $(0, b)$ has degree 1.
\end{proof}

\begin{lem}\label{lem: degree of root}
    Suppose $F$ has a separable irreducible factor of degree $r$.  Then, $r\mathbb{N} \subseteq \mathcal{D}(C/K)$.
\end{lem}
\begin{proof}
    Let $\alpha$ be a separable root of degree $r$ and let $\ell$ be the multiplicity of $\alpha$ as a root of $F(x)$.  Since $F$ is assumed to be free of $q$-th powers, let $i$ be a lift of $\ell^{-1} \pmod q$. Consider the isomorphic curve given by $y^q = (F(x))^i$. After eliminating $q$-th powers from $(F(x))^i$, the resulting isomorphic curve is smooth at $(\alpha, 0)$, showing that $\mathcal{D}(C/K)$ contains $\deg \alpha$.
\end{proof}

\subsection{Proofs of main theorems}

\begin{proof}[Proof of \zcref{thm: main thm}]
    By \zcref{lem: degree set is dense}, $\mathcal{D}(C/K) = \mathcal{D}(U/K)$. Let $P \in U(\overline{K})$ and consider $x(P)$. As the collection $\{X_\mathfrak{s}\}$ covers $\mathbb{A}^1(\overline{K})$, there exists some $\mathfrak{s}$ such that $x(P) \in X_\mathfrak{s}$. If $\mathfrak{s}$ is Galois invariant, then \zcref{lem: no points not zero mod q} shows that either $q$ divides the degree of $x(P)$, or that one of the conditions \ref{condition: bS not 0 mod q first lemma}-\ref{condition: F(gamma) mod q first lemma} holds. By \zcref{lem: congruence conditions imply N>r}, if one of the conditions \ref{condition: bS not 0 mod q first lemma}-\ref{condition: F(gamma) mod q first lemma} hold, then the degree set contains $\mathbb{N}_{\geq r_0}$ and is cofinite. Thus, $\mathcal{D}(C/K)$ is not cofinite if and only if the conditions \ref{condition: bS not 0 mod q first lemma}-\ref{condition: F(gamma) mod q first lemma} do not hold. If $\mathfrak{s}$ is not Galois invariant, the degree of $x(P)$ is a multiple of $|\mathcal{O}(\mathfrak{s})|$, and $\deg_K(P)$ is a multiple of the degree of $x(P)$, so that $\deg_K(P) \in |\mathcal{O}(\mathfrak{s})|\mathbb{N}$.
\end{proof}

\begin{thm}\label{thm: alg closed examples even too}
    Let $D$ be a positive integer divisible $q$, and let $\{n_k\} \subseteq \mathbb{N}$ be a finite sequence such that $q \left( \sum_k n_k \right) = D$. Suppose that $q$ is odd, or that the number of times $1$ occurs in the sequence $\{n_k\}$ is even. Then, there exists a superelliptic curve $C/K$ with equation $y^q = F(x)$, where $D = \deg F(x)$, such that
    \begin{align*}
        \mathcal{D}(C/K) = q\mathbb{N} \cup \bigcup_{n_k>1} n_k \mathbb{N}.
    \end{align*}
\end{thm}

The added condition when $q$ is even is necessary, as when the residue field of $K$ is algebraically closed, \cite{Degrees-on-Varieties}*{Lemma 5.8} shows that there does not exist a curve $y^2 = F(x)$ such that $\mathcal{D}(C/K) = 2\mathbb{N}$ with $\deg F(x) \equiv 2 \pmod{4}$.

\begin{proof}
    We begin with the case where $q$ is odd.  Define $m_k = \{ n_k \mid n_k >1\}$, and reorder so that both $n_k$ and $m_k$ are weakly increasing. We construct a superelliptic curve $C/K$ with equation $y^q = F(x)$ such that
    \begin{align*}
        \mathcal{D}(C/K) = q\mathbb{N} \cup \bigcup_{n_k> 1} n_k \mathbb{N} = q\mathbb{N} \cup \bigcup_{k} m_k \mathbb{N}.
    \end{align*}
    Define $c \coloneqq D - q\sum_{k}m_k$ and let $a$ be relatively prime to $c$, congruent to $1 \pmod q$, and less than $-c$. As $D$ is a multiple of $q$, so is $c$. Define $F(x)$ by
    \begin{align}\label{eq: definition of F(x)}
        F(x) \coloneqq \pi (x^c - \pi^a)\prod_{n_k>1} \prod_{i=1}^q (x^{n_k} - \pi)
    \end{align}
    and consider the curve $C/K$ defined by $y^q = F(x)$. By \zcref{lem: contains qN} and \zcref{lem: degree of root}, we have that $q\mathbb{N} \cup \bigcup_{n_k > 1} n_k \mathbb{N} \subseteq \mathcal{D}(C/K)$.  We show that $\mathcal{D}(C/K) \subseteq q\mathbb{N} \cup \bigcup_{n_k > 1} n_k \mathbb{N}$. Let $P \in U(\overline{K})$ and let $\mathfrak{s}$ be such that $x(P) \in X_{\mathfrak{s}}$. By \zcref{lem: types of S},  $X_{\mathfrak{s}} \subseteq B_{|\alpha|}(\alpha)$ or there exists $\alpha_\text{min}$ such that ${\mathfrak{s}}$ indexes the set $\{ \alpha \mid \val(\alpha) \geq \val(\alpha_\text{min}) \}$. 
    
    Consider the first case. By construction, the roots of $F(x)$ have valuations $\frac{a}{c}$ or $\frac{1}{n_k}$, and as $x(P) \in B_{|\alpha|}(\alpha)$, $\val(x(P)) = \frac{a}{c}$ or $\val(x(P)) = \frac{1}{n_k}$. Let $d$ denote the denominator of $\val(x(P))$. We have a chain of divisibility conditions: $d$ divides $[K(x(P)):K]$, which divides $[K(P):K]$, and as $d = n_k$ or $q \mid c = d$, $[K(P):K] \in q \mathbb{N} \cup \bigcup_k n_k \mathbb{N}$.
    
    In the second case, there exists $\alpha_\text{min}$ such that ${\mathfrak{s}}$ indexes the set $\{ \alpha \mid \val(\alpha) \geq \val(\alpha_\text{min}) \}$. We see that ${\mathfrak{s}}$ is Galois invariant. Note that $F(x)$ may have wildly ramified roots, so that the assumptions of \zcref{lem: no points not zero mod q} fail. For this specific $\mathfrak{s}$, however, all $P \in U(\overline{K})$ fall into case \ref{case: pigeonhole} or case \ref{case: valuation of root}, and the proof proceeds as written provided we verify the conclusion of \zcref{lem: cS integer} (that $c_\mathfrak{s}$ is an integer). Let $k_0$ be the largest index such that $\frac{1}{m_{k_0}} < \val(\alpha_\text{min})$. We compute
    \begin{align*}
        c_{\mathfrak{s}} &= \val(\pi) + \sum_{\beta \in \mathfrak{s}^c} \val(\alpha_{\min} - \beta) = 1 + \sum_{\beta \in \mathfrak{s}^c} \val(\beta) \\
        &= 1 + c \left( \frac{a}{c} \right) + qm_1 \left( \frac{1}{m_1} \right) + \ldots + qm_{k_0} \left( \frac{1}{m_{k_0}} \right) \in \mathbb{Z}.
    \end{align*}
    As $a \equiv 1 \pmod q$, $c_{\mathfrak{s}} \equiv 2 \pmod q$, and as $q \geq 3$, $c_{\mathfrak{s}} \not \equiv 0 \pmod q$. We compute $|\mathfrak{s}|$. The roots of $F(x)$ have valuations $\frac{a}{c} < \frac{1}{m_1} \leq \frac{1}{m_2} \leq \cdots$.
    The number of roots of each valuation is a multiple of $q$, and thus the number of roots $\alpha$ such that $\val(\alpha_\text{min}) \leq \val(\alpha)$ is a multiple of $q$. Thus, $|\mathfrak{s}| \equiv 0 \pmod q$.
    Finally, we check that $\val(F(0)) \not \equiv 0 \pmod q$, which is necessary for case \ref{case: valuation of root}. Let $k_1$ be the largest index in the sequence $\{m_k\}$. By construction, $\val(F(0)) = \val(\pi^{a+1} \pi^{qk_1}) \equiv a + 1 \pmod q$
    and as before, $1 + a \not \equiv 0 \pmod q$.
    As $|\mathfrak{s}| \equiv 0 \pmod q$, $c_\mathfrak{s} \not \equiv 0 \pmod q$, and $\val(F(0)) \not \equiv 0 \pmod q$, we conclude that for all $P \in U(\overline{K})$ such that $x(P) \in X_\mathfrak{s}$, $q \mid \deg P$, as in (the contrapositive of) \zcref{lem: no points not zero mod q}.
    
    When $q$ is 2, we construct $F(x)$ as in \zcref{eq: definition of F(x)}, with slight modifications. Note that $c_{\mathfrak{s}} \equiv 1 + a \pmod 2$, and thus if we want $2 \nmid c_{\mathfrak{s}}$, we must have $a \equiv 0 \pmod 2$. Thus, choose $a$ such that $a \equiv 2 \pmod 4$ and is smaller than $-cn_k$ for all $k$. Then, as
    \begin{align*}
        c = D - 2\sum_{n_k > 1} n_k = 2 \sum_{n_k = 1} n_k
    \end{align*}
    and the number of $k$ such that $n_k = 1$ is even by assumption, $c \equiv 0 \pmod 4$, so that $\frac{a}{c}$ (in lowest terms) has even denominator. The rest of the proof is similar to when $q$ is odd.
\end{proof}

\begin{proof}[Proof of \zcref{thm: nice generate examples theorem}]
    If there exists $k$ such that $n_k = 1$, pick a superelliptic curve $C/K$ with a rational point. Otherwise, the proof is immediate from \zcref{thm: alg closed examples even too}.
\end{proof}

\section{Examples}

Our results also provide a method for computing degree sets.  Although our lemmas assume algebraically closed residue field, this example is computed over $\mathbb{Q}_7$, as the points constructed of minimal degree (i.e., do not arise from having degree a multiple of a degree already in the degree set) are totally ramified, so the degree set is unchanged over $\mathbb{Q}_7^{\text{unr}}$.

\begin{example}\label{ex: computed degree set}
    Let $\alpha = 7^{\frac{1}{2}} + 7^{\frac{3}{5}}$, $\beta = 2\cdot7^{\frac{1}{2}} + 7^{\frac{3}{5}}$, and $\gamma = 3 \cdot 7^{\frac{1}{2}} + 7^{\frac{3}{5}}$ as elements of $\overline{\mathbb{Q}_7}$, and let $f_1(x)$, $f_2(x)$, and $f_3(x)$ be the minimal polynomials of $\alpha$, $\beta$, and $\gamma$ respectively over $\mathbb{Q}_7$. Let $F(x) = 7 f_1(x)f_2(x)f_3(x)$, and let $C/\mathbb{Q}_7$ be the curve given by $y^3 = F(x)$. Then, we have
    \begin{align*}
        \mathcal{D}(C/\mathbb{Q}_7) = 3\mathbb{N} \cup 2\left( \{8,11,13,14\} \cup \mathbb{N}_{>15}\right) \cup 10\mathbb{N} \subseteq 3\mathbb{N} \cup 2 \mathbb{N}.
    \end{align*}
\end{example}
\begin{proof}
   Fix a primitive 5th root of unity $\zeta_5$ and let $\mathfrak{s}$ be a cluster for $F(x)$.  By \zcref{lem: types of S}, all roots of $\mathfrak{s}$ are contained in $B_{|\alpha|}(\alpha)$, or there exists $\alpha_\text{min}$ such that $\mathfrak{s} = \{ \alpha \mid \val(\alpha) \geq \val(\alpha_\text{min})\}$.
    
    We first consider the case where there exists $\alpha_\text{min}$ such that $\mathfrak{s}=\{ \alpha \mid \val(\alpha) \geq \val(\alpha_\text{min})\}$. As all roots of $F(x)$ have the same valuation, $\mathfrak{s}$ must contain all roots of $F(x)$. Then, $|\mathfrak{s}| = 30$ and $c_{\mathfrak{s}} = 1$, and hence we have that
   \begin{align*}
        \val(y(P)) = \frac{30}{3}\val(x(P) - \alpha) + \frac{1}{3}
   \end{align*}
   so that by \zcref{lem: sol in interval}, $3$ divides the denominator of $\val(y(P))$, and hence $\deg P \in 3\mathbb{N}$.
   Thus, we can suppose that all roots of $\mathfrak{s}$ are contained in $B_{|\alpha|}(\alpha)$. Each $B_{|\alpha|}(\alpha)$ contains exactly $5$ roots of $F(x)$, and as all roots in $B_{|\alpha|}(\alpha)$ are equidistant, either $|\mathfrak{s}| = 5$ or $|\mathfrak{s}| = 1$.
   
   If $|\mathfrak{s}| = 1$, then there exists $\alpha \in \mathfrak{s}$ such that $x(P)$ is closer to $\alpha$ than all Galois conjugates of $\alpha$, and Krasner's lemma implies that $10 \mid \deg x(P) \mid \deg P$, showing that $\deg P \in 10\mathbb{N}$. As the roots of $F(x)$ have degree 10, by \zcref{lem: degree of root}, $\mathcal{D}(C/K)$ contains $10\mathbb{N}$.

   Suppose $|\mathfrak{s}| = 5$. We detail only the case where $\mathfrak{s} = \{7^{\frac{1}{2}} + \zeta_5^i7^{\frac{3}{5}}\}_{i = 1}^{5}$, as the other cases are similar. We have that $c_{\mathfrak{s}} = \frac{27}{2}$. Fix $\alpha \in \mathfrak{s}$.
   By \zcref{lem: types of S}, we have that $X_{\mathfrak{s}} \subseteq B_{|\alpha|}(\alpha)$, and applying \zcref{lem: tamely ramified generalized krasners}, we find that for all $\eta \in B_{|\alpha|}(\alpha)$, $K(7^{\frac{1}{2}}) \subseteq K(\eta)$, so that if $x(P) \in X_{\mathfrak{s}}$, then $K(7^{\frac{1}{2}}) \subseteq \textbf{k}(x(P))$. We extend to $K(7^{\frac{1}{2}})$ and renormalize the valuation, remembering that we have divided the degree of our point by 2.  
   Following \zcref{lem: congruence conditions imply N>r} over $K(7^{\frac{1}{2}})$, we have that $\gamma_\mathfrak{s} = 7^{\frac{1}{2}}$ and $(M, N) = (1, \frac{6}{5})$. If $\val(x(P) - \gamma_\mathfrak{s}) = \frac{a}{r}$, then
   \begin{align*}
       \val(y(P)) = \frac{5}{3} \cdot \frac{a}{r} + \frac{27}{3} = \frac{5a + 27r}{3r}
   \end{align*}
   and so 3 divides the denominator of $\val(y(P))$ unless $5a + 27r \equiv 0 \pmod{3}$, which occurs if and only if $a \equiv 0 \pmod{3}$. We search for $r \in \mathbb{N}$ that produce $\frac{a}{r} \in (1, \frac{6}{5})$ such that $a \equiv 0 \pmod 3$. For $r$ large enough, by \zcref{lem: sol in interval}, this condition is always satisfied, and a computer finds the condition is satisfied for $8$, $11$, $13$, $14$, and all integers greater than $15$. 
\end{proof}

\begin{remark}\label{rem: dokchitser}
The code of \cite{Regular-Models} applied to \zcref{ex: computed degree set} gives the following diagram.

$$
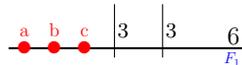
\begin{tikzpicture}[xscale=0.8,yscale=0.7,
  l1/.style={shorten >=-1.3em,shorten <=-0.5em,thick},
  l2/.style={shorten >=-0.3em,shorten <=-0.3em},
  lfnt/.style={font=\tiny},
  rightl/.style={right=-3pt,lfnt},
  mainl/.style={scale=0.8,above left=-0.17em and -1.5em},
  facel/.style={scale=0.5,blue,below right=-0.5pt and 6pt},
  redbull/.style={red,label={[red,scale=0.6,above=-0.17]#1}}]
\draw[l1] (0.00,0.00)--(2.96,0.00) node[mainl] {6} node[facel] {$F_1$};
\node[redbull=a] at (0.00,0.00) {$\bullet$};
\node[redbull=b] at (0.50,0.00) {$\bullet$};
\node[redbull=c] at (1.00,0.00) {$\bullet$};
\draw[l2] (1.50,0.00)--node[rightl] {3} (1.50,0.66);
\draw[l2] (2.30,0.00)--node[rightl] {3} (2.30,0.66);
\end{tikzpicture}
$$

Note that this method is unable to fully compute the special fiber, and thus is insufficient for computing the degree set in this case.  Using the main result of \cite{Degrees-on-Varieties} and the fact that \zcref{ex: computed degree set} will have the same degree set over $\mathbb{Q}_7^\text{unr}$, we see that every component of the special fiber must have multiplicity divisible by 2 or 3.

\end{remark}

\bibliographystyle{alpha}
\bibliography{refs}

\end{document}